\documentclass[a4paper,11pt,reqno]{amsart}
\usepackage{graphicx}
\usepackage[latin1]{inputenc}
\usepackage{mathrsfs}
\usepackage{dsfont}
\usepackage{hyperref}
\usepackage{amsmath}
\usepackage{amssymb}
\usepackage{amsthm}
\usepackage{amsfonts}
\usepackage{amstext}
\usepackage{amsopn}
\usepackage{amsxtra}
\usepackage{mathrsfs}
\usepackage{dsfont}
\usepackage{esint}
\usepackage{enumitem}

\newcommand{\R}{{\mathbb R}}

\newcommand{\N}{{\mathbb N}}
\newcommand{\Z}{{\mathbb Z}}

\renewcommand{\(}{\left(}
\renewcommand{\)}{\right)}

\newtheorem{thm}{Theorem}

\newtheorem{lemma}[thm]{Lemma}

\newtheorem{prop}[thm]{Proposition}
\newtheorem{cor}[thm]{Corollary}
\newtheorem{remark}[thm]{Remark}

\renewcommand\phi{\varphi}

\renewcommand{\tilde}{\widetilde}
\newcommand{\finprf}{\null\hfill$\square$\vskip 0.3cm}

\newcommand{\bq}{\begin{equation}}
\newcommand{\eq}{\end{equation}}

\usepackage{color}

\renewcommand{\(}{\left(}
\renewcommand{\)}{\right)}

\newcommand{\be}[1]{\begin{equation}\label{#1}}
\newcommand{\ee}{\end{equation}}
\renewcommand{\(}{\left(}
\renewcommand{\)}{\right)}

\title[Gagliardo-Nirenberg-Sobolev inequalities on planar graphs]{Gagliardo-Nirenberg-Sobolev inequalities on  planar graphs}

\author[M.J. Esteban]{Maria J. Esteban}
\address{CEREMADE, CNRS, Université Paris-Dauphine, PSL Research University, Place de Lattre de Tassigny, 75016 Paris, France} 
\email{esteban@ceremade.dauphine.fr}

\date{\today}

\DeclareRobustCommand{\SkipTocEntry}[5]{}

\begin{document}

\begin{abstract}
In this paper we study a family of interpolation Gagliardo-Nirenberg-Sololev inequalities on planar graphs. We are interested in knowing when the best constants in the inequalities are achieved. We also analyse the set of solutions of the corresponding Euler-Lagrange equations. 
\end{abstract}

\maketitle


In the past years many works have been devoted to the study of functional inequalities, their best constants, the existence of extremal functions and their qualitative properties. There are examples for compact and not compact manifolds, in the flat Eucledian space in various dimensions, on the line, etc. Among these inequalities, the so-called Gagliardo-Nirenberg-Sobolev inequalities: if $D$ is an open domain of $\R^d$, a very simple class of inequalities has the form 
\be{init-inf}
\(\int_D  |v|^p\)^{2/p} \le C \int_D |\nabla v|^2 + \lambda |v|^2 \;, \;\mbox{ for all }\; v\in H^1_0(D)\,,
\ee
with $\lambda\in \R_+$, $p\in  (2, 2^*]$,  $2^*=2d/(d-2)$ if $d\ge 3$, $2^*= +\infty$ if $d=1,2$.

These inequalities have been studied throughly and much is known about their best contacts, or their extremal functions, for a large family of domains $D$. In the case $D= S^d$,  $d\ge 1$, $p>2$, these inequalities, in their optimal form, are usually attributed to W.~Beckner~\cite{MR1230930} but can also be found in~\cite[Corollary~6.1]{BV-V}. However an earlier version corresponding to the range $p\in[1,2)\cap(2,2^\#)$ was established in the context of continuous Markov processes and linear diffusion operators by D.~Bakry and M.~Emery in~\cite{Bakry-Emery85,MR808640}, using the \emph{carr\'e du champ} method, where~$2^\#$ is the \emph{Bakry-Emery exponent} defined as
$
2^\#=\frac{2\,d^2+1}{(d-1)^2}
$, 
for any $d\ge2$. See also \cite{Do-Es-La-2014}. For the case of general compact manifolds, see \cite{MR1338283, MR1631581, MR1412446, MR1435336, DolEstLoss-compactmanifolds}. In the case of the line $\R$, see \cite{Dol-Est-Lap-Loss-2014}. For the case of the whole domain $\R^d$, there are many references. See for instance Appendix B in \cite{Do-Es-La-2014} and references therein for details about the best constants and optimizers.

Variations of these inequalities corresponding to the inclusion of external magnetic fields, and thus replacing the gradient operator by the magnetic gradient operator, also have been dealt with, with less success for the characterization of the best constants or the extremals. To our knowledge, very little has been done in this direction concerning graphs. Some results about inequalities on graphs can be found in~\cite{MR3988826}. The goal of this paper is to analyze the above inequalities when $D$ is a  metric planar graph. The same results can be extended to $d$-dimensional graphs, with $d\ge 3$,  almost straightforwardly, since the nature of the graphs is $1$-dimensional.

In this paper we will deal only with locally finite graphs, that is, graphs for which there is no accumulation of edges  at any vertex of the graph.

Assume that $G$ is a graph consisting of  vertices $\{V_i\}$ and  edges $\{e_j\}$ connecting the vertices, their endpoints. We will denote $E$ the set of edges of the graph $G$. Of course, each bounded edge $e$ can be identified as a pair of vertices $\{V_\ell, V_k\}$, and an unbounded edge will be defined by one endpoint and a direction (a vector).  In $\R^d$ endowed with the usual Lebesgue measure,  such a graph is a {\sl metric} graph, because to each edge we can assign a given length $\ell$, which is the distance between its endpoints, and this distance can be infinite if the edge is unbounded. As it can be seen for instance in \cite{Kuchment-2008}, one can define $L^p$ and Sobolev spaces on such graphs. For instance, for $p\ge 1$, the space $L^p(G)$  can be defined as the set of measurable functions $f$ satisfying:
$$ 
\| f\|_{L^p(G)}^2 = \sum_{e\in E} \| f\|_{L^p(e)}^2 < +\infty\,,
$$
while the Sobolev space $H^1(G)$ consists of all continuous functions $f$ on $G$ belonging to $H^1(e)$ for each edge $e$, and such that
$$
\| f\|_{H^1(G)}^2 = \sum_{e\in E} \| f\|_{H^1(e)}^2< +\infty\,.
$$
A metric graph $G$ will be called a quantum graph if associated to $G$ there is a Hamiltonian $\mathcal H$ that acts as a second order operator on the edges, and that is accompanied by ``appropriate" vertex conditions. In this paper we will consider the Hamiltonian associated to the sesquilinear form
$$
h[f, g]:= \sum_{e\in E} \int_e (f'\, \overline{g'}+f\, \overline{g})\, dx
$$
on $H^1_0(G)$, which is associated to the Hamiltonian  $ - v'' + v $ endowed with the
 so-called Kirkhoff conditions at the internal vertices of $G$: for each vertex $V$, if $\{e_{k_1}, e_{k_2}, \dots, e_{k_L}\}$  are all the edges sharing $V$ as an endpoint, then,
 $$
\sum_{j=1}^L f_j'(V+)=0\,,
$$
where $f_j'(V+)$ denotes the derivative of $f$ at $V$ in the outward direction towards the edge $e_{k_j}$ at $V$. Here we will work with this Hamiltonian on metric locally finite graphs in the Dirichlet case, that is, in the case where in the outer vertices of the graph,  those belonging only to one edge, the functions $f$ satisfy the boundary condition  $f(v)=0$. 

In this paper we will  analyze the functional inequalities
\be{ineq}
\(\int_G  |v|^p\, dx\)^{2/p}  \le  C_p(G)\,\int_G (|v'|^2 + |v|^2)\, dx \,,  \quad\mbox{for all functions }\; f\in H^1_0(G)\,,
\ee
with  $p \in (2, +\infty)$, where $C_p(G)$ denotes the best constant in the above inequality. The same can be done for other Hamiltonians of the same kind. Our choice of Hamiltonian is done in order to be able to provide a set of results providing  precise information about the value of the best constants and the extremal functions in the inequalities. The more complicated the Hamiltonian, the more difficult this task will be.

Note that the best constant $C_p(G)$ can also be defined by
\be{min}
\frac{1}{C_p(G)} = \inf_{v\in H^1(G)}\;  \frac{\int_G (|v'|^2 + |v|^2)\, dx }{\(\int_G  |v|^p\, dx\)^{2/p}}\,,
\ee
and a function $v$ is a minimizer for \eqref{min} if and only if it is a extremal function for \eqref{ineq}. Moreover, since in any given open domain $D$,  for all functions $v\in H^1(D)$, $ \int_D | |v|'|^2\,dx \le \int_D |v'|^2\,dx$, it is trivial to see that there is always a non-negative minimizer for the above problem. Moreover, because of the regularity of solutions of the ODE on any interval, there is no minimizer which changes sign in the middle of one of the edges. But there could be minimizers which are positive on some edges and negative on others; in that case they must be equal to $0$ in the vertices separating them, of course. In the sequel we will then consider only non-negative extremal functions.

For simplicity of notation, let us define the functional $F_{G, p}(f)$ involved in the above minimization problem by
$$F_{G, p}(f) = \frac{\int_G (|v'|^2 + |v|^2)\, dx }{\(\int_G  |v|^p\, dx\)^{2/p}}\,.
$$

\begin{prop}\label{Prop:ODE}  Note that up to multiplication by a constant, a  minimizer for \eqref{min} satisfies the set of ODEs
\be{ODE}
-v'' + v = |v|^{p-2}v  \;\mbox{ in }\; e,\quad \forall \,e\in E\,. 
\ee
together with the boundary condition $v=0$ at the outer vertices of $G$ and the Kirkhoff conditions at all inner vertices of $G$. Moreover, if $v$ is an extremal function for \eqref{ineq}, i.e., a solution to \eqref{min},   we have
\be{otherdef}
\frac{1}{C_p(G)} =  \| v\|_{ L^p(G)}^{p-2} \,.
\ee
\end{prop}
\noindent {\bf Definition. }
We will say that a function $v\in H^1_0(G)$ satisfying \eqref{ODE} and the Kirkhoff  conditions at all inner vertices of $G$ is a solution to the Kirkhoff-ODE system in $G$.

 \medskip
In this paper we analyse the value of the constants $C_p(G)$, and in particular whether they are achieved or not, that is,  whether there exist  solutions to the minimization problem \eqref{min} or not. In parallel, we will investigate the existence and multiplicity of solutions to the Kirkhoff-ODE system in $G$. We will prove some general results, and analyse a set of examples for which we can make these  results more precise.

In some cases,  depending of the structure of the graph, we will be able to prove the existence of solutions to  \eqref{min}. In some other cases, the non-existence of solutions to the Kirkhoff-ODE system in $G$ will automatically imply the non-existence of minimizers for \eqref{min}. In other cases we will only prove that the existence of solutions to \eqref{min}, that is, of extremal functions for \eqref{ineq}, depends on a strict inequality between the infima of two related energy functionals, but assessing that strict inequality will not be easy in most cases. Although the problem studied in this paper is very simple,  a large variety of possible situations and results can arise.

In Section \ref{Sec:phase-portrait}  we will analyse the easiest cases, when $G$ is the line $\R$ or the half-line $\R_+$. Section \ref{Sec:existence} contains a number of general results about the attainability of the best constants $C_p(G)$ in \eqref{ineq}, that is, the existence of minimizers for \eqref{min}. In Section \ref{Sec:Dirichlet} we will describe in detail several classes of bounded and unbounded graphs, for which we explore not only the attainability of the best constants $C_p(G)$, but also the existence and multiplicity of solutions to the Kirkhoff-ODE system in $G$.

\bigskip
\section{The phase plane, the line and the half-line}\label{Sec:phase-portrait}

In order to understand well the results that  will be stated and proved in this paper, let us first discuss the phase plane of the dynamical system associated to the ODE equation \eqref{ODE} in $\R$.  In that phase plane there are two centers, the points $(\pm1, 0)$; an instable point, the origin; two heteroclinic orbits, one corrresponding to a positive solution and the other one, to a negative solution;  periodic orbits corresponding to either positive or negative functions, which turn around the two centers; and finally other periodic orbits corresponding to sign changing solutions which turn around the origin. For more information about the periodic solutions see \cite{Yagasaki-2013, Ben-Des-Loss-2016} and references therein. See also Section 5.2 in~\cite{MR3352243} for a similar phase plane analysis. Concerning the explicit heteroclinic orbits, see \cite{Weinstein-1986, Dol-Est-Lap-Loss-2014} for instance. See Figure  \ref{Fig:phase-portrait} for a graphical representation of the phase plane.

\begin{figure}[ht]\label{Fig:phase-portrait}
\begin{center}
\includegraphics[width=6cm,height=4cm]{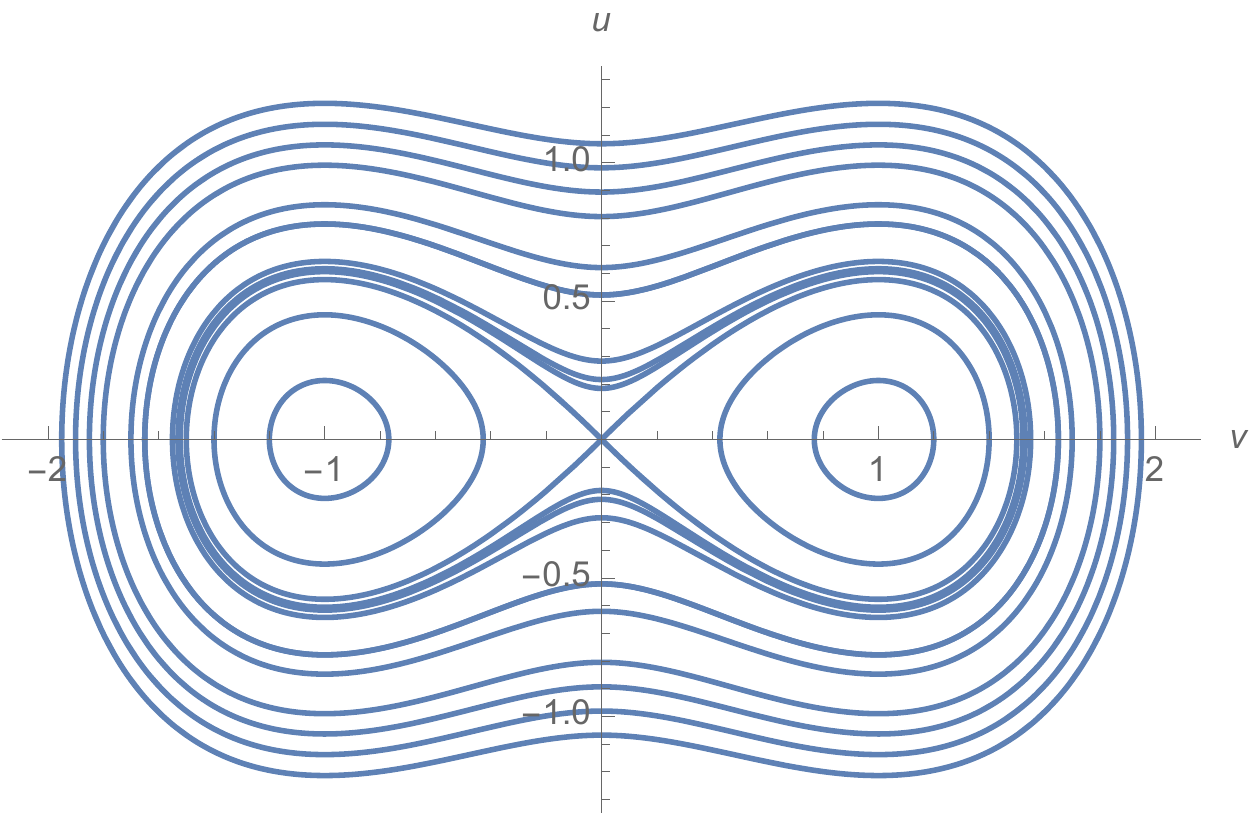}
\caption{ Phase plane for the dynamical system $ v'=u,\; -u'+v= |v|^{p-2}v\;\; \mbox{in}\;\;\R$, \mbox{ with  }  $p=3$  }
\end{center}
\end{figure}

Below we  discuss the cases of the line and the half-line, while in the next sections we will examine the cases of star-graphs, that is, sets of half-lines joining at the origin and other more complicated graphs.

The case of the line $\R$ is well-known, and corresponds to one of the heteroclinic orbits in the above figure, see \cite{Weinstein-1986, Dol-Est-Lap-Loss-2014}.
\begin{prop}\label{Prop:line}
Let $p>2$. Then, the unique non-negative solution  of the equation
\be{solution-R}
-v'' + v = |v|^{p-2}v  \;\mbox{ in }\; \R
\ee
is
\be{vbar}
\bar v(s):= \(\frac{p}2\)^\frac1{p-2} \(\cosh\(\frac{(p-2)\,s}2\)\)^{-\frac2{p-2}}
\ee
and 
\[
\frac1{C_p(\R)}= I_{p,0}:= \frac{\int_\R |\bar v'|^2 + |\bar v|^2\, ds}{\(\int_\R  |\bar v|^p\,ds\)^{2/p}} = \(\int_\R|\bar v|^p\)^\frac{p-2}{p}= \(\frac{p}2\)^{\frac2p} \,\( \frac { \sqrt{\pi}\, \Gamma\(2+\frac{2}{p-2}\)}{\Gamma\(\frac12+\frac{p}{p-2}\)}\)^\frac{p-2}{p}\,.
\]
The fonction $-\bar v$ is the unique non-positive solution to \eqref{solution-R}. Moreover, $\pm \bar v$ are the unique solutions to \eqref{solution-R} in the space $H^1_0(\R)$. \end{prop}

The above phase plane shows what happens in the case of the half-line: all solutions of \eqref{ODE} satisfying $v(0)=0$, $v'(0)>0$ are periodic in $\R$, and therefore cannot tend to $0$ at infinity. This means the following.  

\begin{lemma}\label{Lemma:half-line}
Let $p>2$. There is no non-negative solution of the equation
\[
-v'' + v = |v|^{p-2}v  \;\mbox{ in }\; \R_+=(0, +\infty)
\]
belonging to the space $H^1_0(\R_+)$. Therefore, the constant $C_p(\R_+)$ is never achieved.
\end{lemma}
 A different, direct, proof of this result is given below.

\bigskip
\section{Existence and non-existence of extremal functions for \eqref{ineq} or, equivalently, of minimizers for \eqref{min}}\label{Sec:existence} 
 Concerning the existence of minimizers for the minimization problem \eqref{min} for general graphs $G$, there are two very distinct cases depending on the properties of $G$: either the graph is bounded, and in this case the answer is quite direct and positive. Or the graph is unbounded and in this case the situation can be more complicated depending on the graph's structure. The bounded case is very simple, as it can be seen below.

\begin{prop}\label{Prop:compact-ms}
Let $p>2$.  Assume that the graph $G$ is locally finite and bounded. There always exists at least one solution of the minimisation problem \eqref{min}, and therefore the constant $C_p(G)$ is always achieved in \eqref{ineq}.
\end{prop}
\begin{proof}
Let  $ \{v_n\}_n$ a minimizing sequence for \eqref{min} that we normalize to satisfy $\int_{G} |v_n|^p\,dx =1$ (without loss of generality). Trivially, the sequence  $ \{v_n\}_n$ is bounded in $H^1_0(G)$ and since this space is compactly embedded in $L^p(G)$, up to subsequences, $ \{v_n\}_n$ is compact in $L^p(G)$ and converges, weakly in $H^1_0(G)$ and strongly in $L^p(G)$,  to some $v\in H^1_0(G)$ satisfying
$$
\int_{G} |v|^p\,dx =1\,.
$$
By the lower  semicontinuity of the $H^1_0(G)$ norm,
 $v$ is a minimizer for \eqref{min}.
\end{proof}

 Let us next give some definitions that will be useful in the sequel of the paper.
 
 To deal with the case of  locally finite unbounded graphs, let us first define a  {\it chain} $C$ as a  sequence of edges $\{e_i\}_{i= 1, \dots, N}$, $N\ge1$ and possibly $N=\infty$,  of lengths $\ell_i$  and vertices $V^\pm_i$ such that if $e_i= [V^-_i, V^+_i]$ and for all $i\ge 1$, $V^+_i= V^-_{i+1}$. Note that if the extremal edges of the chain are unbounded, then, $V^-_1$ and $V^+_N$ could find themselves  at infinity, and the corresponding lengths would be equal to  $+\infty$.
 
If the chain is bounded, a function $v\in H^1_0(C)$ can be identified with a function $w_x$ in $H^1_0(x, x+\sum_i \ell_i)$ for any $x\in \R$ and there is equality for the  $L^q$ and $H^1_0$ norms of $v$ and $ w_x$. If the chain is semi-bounded (resp. unbounded in the two directions), then the identification can be done with $\R_+$ (resp. $\R$). This means that the {\it angles} between the edges of a chain do not play any role in the computation of those norms. In the above cases we will speak of {\it equivalence} of the chain $C$ with  the interval $(x, x+\sum_i \ell_i)$,  $\R_+$ or  $\R$.  
  
A consequence of the above is  that if a chain $C$  is semi-bounded, then, $C_p(C) = C_p(\R_+)$, while if  it is unbounded on both sides, then, $C_p(C) = C_p(\R)$. 
  
 In the case of unbounded graphs, the most general result that we can prove about \eqref{min} is the following.
 
 \begin{thm}\label{Thm:conc-comp}
 Let $p>2$.  Assume that $G$ is a locally finite unbounded graph. Then, except in the trivial case when $G$ is a chain ``equivalent" to $\R$, all minimizing sequences for the minimization problem \eqref{min} are either relatively compact in $H^1_0(G)$ or there is a function $v\in H^1_0(G)$ and a sequence $\{x_n\}_n\in G$ such that
 $$
 |x_n| \longrightarrow_{_{\hspace{ -4mm}  n}} \hspace{3mm} +\infty\,,      
 $$
 for all $\epsilon>0$ there exists $R>0$ such that
 $$
 \int_{B_R(x_n)\cap G} |v(x)|^p\, dx \ge 1- \epsilon\,,
 $$ 
 and
 $$
 \| v_n - v\|_{H^1_0(G)} \longrightarrow_{_{\hspace{ -4mm}  n}} \hspace{3mm}  0\,.
 $$
 \end{thm}
 \begin{proof}
 We will not prove this theorem in detail, since its proof is based on very well-known concentration-compactness arguments that can be found, for very similar problems, in \cite{Lions-cc-compact1, Lions-cc-compact2}. Section I.2 in \cite{Lions-cc-compact2} treats a minimization problem very similar to \eqref{min}.
 
 Let $\{v_n\}_n$ a minimization sequence for \eqref{min}. By concentration-compactness, if we consider the sequence of $L^1(G)$ functions $\rho_n:= |v_n|^p$, it is easy to see that this sequence cannot ``vanish", in the sense that it cannot verify that  for all $R>0$,
 $$
 \lim_{n \rightarrow \,+\infty} \;\sup_{x\in G} \int_{ B_R(x)\cap G} \rho_n\, dx =0\,.
 $$
 Indeed, this follows from an easy application of Lemma I.1 in \cite{Lions-cc-compact2}. Next, the so-called ``dichotomy" defined in \cite{Lions-cc-compact1} cannot hold either, as this is forbidden by the homogeneity properties of the  functional $F_{G, p}$ which defines the problem \eqref{min}. Finally, since in dimension $1$ any exponent $2<p<\infty$ is subcritical (that is, $H_0^1(G)$ embeds locally compactly into $L^p(G)$), up to subsequences, there exists a sequence of points in $G$, $\{x_n\}_n$ and a function $v\in H^1_0(G)$  such that in the norm of $H^1_0(G)$, the functions $v_n$ are very close to the function $v$ restricted to a bounded neighborhood of $x_n$, the $L^p$ norm of $v$ on those bounded neighborhoods of $x_n$ being as close to $1$ as desired.  \end{proof}
 
 \begin{remark}Interesting locally finite unbounded graphs are those which are periodic in the direction in which they are unbounded. Consider for instance  the graph $G$ which is the union of $\,\R$ and a fixed set of bounded edges attached to the points $(n,0)$ on $\R$. For this kind of graphs there is always existence of a solution to \eqref{min} because we can just apply the above theorem and by periodicity, we can always consider that up to translation, any minimizing sequence is close to a function in $H^1_0(G)$ translated by a bounded sequence of points $\{(y_n, 0)\}_n$. Indeed, it is enough to define $(y_n,0) := x_n - (k_n n, 0)$ with $k_n\in \Z$ and $|y_n|\le 2$.
\end{remark}
 
 Let us next study some other particular cases regarding the existence of minimizers for the minimization problem \eqref{min}.
 
 \begin{lemma}\label{constant-ineq} For all $p>2$ and for any locally finite graph containing an  unbounded edge, we have
\[
C_p(G) \ge C_p(\R)\,.
\]
\end{lemma}
\begin{proof}
If $G$ is not equivalent to $\R$, let us consider a semi-bounded edge in $G$,  $e_0$ equivalent to $\R_+$. Rotate $G$ such that $e_0$ coincides with $\R_+$. For all $n\in \N$, let $P_n\in e_0$ be such such that $|P_n| \to +\infty$. 
Let us define $u_n$ as the product of the function $\bar v(x-P_n)$ restricted to $e_0$ and a function $\eta\in C^1(\R_+)$ which satisfies  $\eta(0)=0, \,\eta(x)=1$ for all $|x|\ge 1$, with $\bar v$ as in \eqref{vbar}. Define $u_n \equiv 0$ on the other edges of $G$. The function $u_n$ belongs to $H^1_0(e_0)\cap H^1_0(G)$ and is supported in $e_0$. Moreover,
$$
F_{G, p}(u_n) \longrightarrow_{_{\hspace{ -4mm}  n}} \hspace{3mm}   F_{\R, p}(\bar v)= \frac{1}{C_p(\R)}\,.
$$
Since
$
\frac{1}{C_p(G)}  \le F_{G, p}(u_n)
$
for all $n$,  $C_p(G) \ge C_p(\R)$.
\end{proof}

An easy application of concentration-compactness, and in particular, of Theorem \ref{Thm:conc-comp}, allows to prove the following result.
\begin{thm}\label{Thm:conc-comp2}
Let $p>2$.  Assume that $G$ is a locally finite unbounded graph  that  is unbounded only along half-lines or chains equivalent to half-lines. Then, all minimizing sequences for \eqref{min} are relatively compact in $H^1_0(G)$ if and only if 
\be{strictt}
C_p(G) > C_p(\R)\,.
\ee
and when the above strict inequality holds true, there is at least one solution to \eqref{min}.
\end{thm}
\begin{remark}
Note that some unbounded graphs excluded from the above theorem are those for which one can ``escape" to infinity by following chains through edges which are not equivalent to half-lines. For instance,  infinite {\it regular} trees, that is, trees for which all vertices with equal distance to the root of the tree have same branching number (number of edges joining at that vertex) and for which the edges emanating from those vertices are more than $1$ and have same length.
\end{remark}
The following result is an easy corollary of the above theorem, Proposition \ref{Prop:ODE} and Lemmata \ref{Lemma:half-line} and \ref{constant-ineq}.
\begin{cor}\label{Cor:half-line-energy}
If a graph $G$ is a half-line or a chain equivalent to it, 
\[
C_p(G)= C_p(\R)\,.
\]
\end{cor}
\noindent{\bf Proof of Theorem \ref{Thm:conc-comp2}}
With  $p>2$, let $\{v_n\}_n$ be a minimizing sequence for \eqref{min} that we normalize so that for all $n$, $\int_G |v_n|^p\, dx =1$. By Theorem \ref{Thm:conc-comp}, either the sequence $\{v_n\}_n$ is relatively compact, or the sequence $\{x_n\}$ appearing in the proof of this theorem is unbounded. In that case, the function $v$, duly cut to be supported on an unbounded edge, should be a good test function for the problem of minimization in $\R$, since the graphs we are looking at are {\it equivalent} to $\R$ at infinity. But if $
C_p(G) > C_p(\R)$, this cannot happen, because 
$F_{G, p}(v_n) $ would be close to $F_{\R, p}(v) $ as $n$ increases, which is forbidden by  our assumption, since $F_{G, p}(v_n)$ approaches   $1/ C_p(G)$ as $n$ goes to $+\infty$, while $F_{\R, p}(v)$ is larger than  $1/ C_p(\R)$. So, the sequence $\{x_n\}_n$ is bounded and $\{v_n\}_n$ is relatively compact.  Finally, if the strict inequality \eqref{strictt} is not satisfied, that is, if $C_p(G)= C_p(\R)$, then, by taking almost-minimizers of the problem \eqref{min} in $\R$, we can construct minimizing sequences for \eqref{min} on $G$ so that they are not relatively compact, since their ``essential supports" escape to infinity.
\finprf

Next, let us give examples of cases where the graphs are unbounded only along half-lines, but for which the strict inequality $C_p(G) > C_p(\R)$ is not satisfied.

\begin{thm}\label{Thm:one-unbdd}
Let $p>2$.  Assume that the graph $G$ is a half-line $e_0$ plus a finite number of bounded or unbounded edges attached to $e_0$ at a unique inner vertex $P$ (the finite vertex of $e_0$). Then, there is no solution to \eqref{min} and
\be{equality}
C_p(G) =  C_p(\R)\,.
\ee
\end{thm}
\begin{proof} 
Without loss of generality, by a simple translation and rotation, we may assume that $e_0=\R_+$. Let
 $v\in H^1_0(G)$ be a minimizer for \eqref{min}. Then it satisfies the Kirkhoff-ODE system in $G$. Let us then find a chain in $G$, $C$, such that $v\in H^1_0(C)$ and the sum of the outward derivatives of $v$ at the inner vertex of $C$ is non-negative.
 Assume that there are $N$ edges $e_i$ attached to $e_0$ at the inner edge, $P$,  and call $d_i$, $i= 1, \dots, N$ the outward derivatives of $v$ at $P$. By the Kirchoff conditions, we have
\be{Kir}
\sum_{i=0}^N d_i =0\,
\ee
If $N=1$, $C= e_0 \cup e_1$,  $d_0+d_1=0$. Assume that $N>1$ and that  for all $j\ge 1$,  $d_0+d_j <0$. Then $N d_0 + \sum_{i=1}^N d_i <0$, but $N d_0 + \sum_{i=1}^N d_i = (N-1) d_0 +  \sum_{i=0}^N d_i$, and therefore, by \eqref{Kir}, $d_0<0$. Moreover, $v$ satisfies $v=0$ at the outer vertex of $e_i$ for all $i=1, \dots, N$. 
Looking at \eqref{Fig:phase-portrait}, it appears clearly that for all $i\ge 1$,  $|d_i|> |d_0|$, since the point $(v(P), d_0)$ lies on the ``positive" heteroclinic orbit, and for $i\ge 1$, all the points $(v(P), d_i)$ lie on the outer periodic orbits around the heteroclinic ones. Hence necessarily, there is $i_0\ge 1$ such that $d_0+d_{i_0}\ge 0$, and then, we will choose $C= e_0\cup e_{i_0}$.

Let us now end the proof of the theorem. By the properties of $C$, 
\[
\hspace{2cm}\frac{1}{C_p(C)} \le F_{C, p}(v) = \frac{\int_{C} \(|v'|^2 + |v|^2\)\, dx }{\(\int_{C}  |v|^p\, dx\)^{2/p}} =  \frac{ \( \int_{e_0} + \int_{e_{i_0}}\)\( |v'|^2 + |v|^2\)\, dx }{\(\(\int_{e_0} +\int_{e_{i_0}} \)|v|^p\, dx\)^{2/p}}\,.
\]
By using Proposition \ref{Prop:ODE}, the way $C$ was chosen and integration by parts, we find
\begin{multline*}
  \frac{1}{C_p(C)} \le  \frac{ \( \int_{e_0} + \int_{e_{i_0}}\) |v|^p\, dx - (d_0+d_{i_0})} { \( \int_{e_0} + \int_{e_{i_0}}  |v_j|^p\, dx\)^{2/p}}  
\le   \frac{  \int_C  |v|^p\, dx } { \( \int_C  |v|^p\, dx\)^{2/p}} = \|v\|^{p-2}_{L^p(C)}
\end{multline*}
By Lemma \ref{constant-ineq} and Proposition \ref{Prop:ODE}, 
\[
\frac{1}{C_p(\R)} \ge\frac{1}{C_p(G)} =  \| v\|_{ L^p(G)}^{p-2}  \ge  \|v\|^{p-2}_{L^p(C)}\,,
\]
and the last inequality above is strict if $G\neq C$. Hence, by Lemma \ref{constant-ineq}
\[
 \frac{1}{C_p(C)}\le   \|v\|^{p-2}_{L^p(C)} < \frac{1}{C_p(G)} \le \frac{1}{C_p(\R)}\,.
 \]
But since $C$ by itself is equivalent to a half-line, by Corollary \ref{Cor:half-line-energy},
\[
\frac{1}{C_p(C)} =\frac{1}{C_p(\R)}\,, 
\]
a contradiction with the above strict inequality.  Therefore there is no solution to the Kirkhoff-ODE system in $G$. Finally, if \eqref{equality} were not true, by Theorem \ref{Thm:conc-comp2} and Lemma \ref{constant-ineq}, there would be a solution to \eqref{min}, and therefore, by Proposition \ref{Prop:ODE}, also a solution to the Kirkhoff-ODE system in $G$. Again a contradiction.
\end{proof}

\bigskip
\section{Analysis of the solutions to the Kirkhoff-ODE system  for various families of locally finite graphs.}\label{Sec:Dirichlet} 
In this section we will explore either the case of star-graphs or the case of graphs with one, two or three edges. Of course, interesting cases arise also for a larger number of edges, but  we will not explore those in detail in this article.

We will analyze both the existence of minimizers for \eqref{min} and the existence of solutions for the Kirkhoff-ODE system. Indeed, as already explained above, the existence of minimizers for \eqref{min} implies the existence of solutions for the Kirkhoff-ODE system. Consequently, the non-existence of solutions for the Kirkhoff-ODE system automatically implies the non-existence of minimizers for \eqref{min}.

In the case when the graph has an unbounded edge, we have proved in Theorem \ref{Thm:one-unbdd} that  there is no solution to \eqref{min}. But in the next sections we will see that there are cases of graphs containing semi-bounded edges for which there are solutions to the Kirkhoff-ODE system, and others for which no solution to the Kirkhoff-ODE system exists.

\subsection{Star graphs}\label{Sec:15-infty}

A star-graph is a graph which is composed of $N$ half lines meeting at some point, for instance the origin. Let us prove next that the only star-graph for which there are solutions to the minimization problem \eqref{min} are those with $N=2$.

\begin{thm}
Let $p>2$. Assume that $G$ is a star-graph with $N$ infinite edges and a single vertex at the origin. If $N$ is even there is an infinity of non-negative (actually positive) solutions to the equation
\[
-v'' + v = |v|^{p-2}v  \;\mbox{ in }\; G
\]
but none of them is minimizing for \eqref{min} (or extremal for \eqref{ineq}) unless $N=2$. When $N=2$ all the non-negative solutions to the above equation are extremals for \eqref{ineq} and minimizers for \eqref{min}.

If $N=1$, no solution to \eqref{ODE} exists as Lemma  \ref{Lemma:half-line} shows. If $N$ is odd and larger than $1$, the unique non-negative solution to the Kirkhoff-ODE system in $G$ is  such that $v(0)= \bar v(0)$ and all the outward derivatives at the origin, in the direction of all the edges of the graph, are equal to $0$. But this solution is not a solution to the minimization problem \eqref{min}, which has no solution.
\end{thm}

\begin{proof}
The case $N=1$ is dealt with by Lemma  \ref{Lemma:half-line}. Assume further that $N>1$. Let $v$ be a solution of the Kirkhoff-ODE system in $G$. Since for all $x\in G$,  the pair $(v(x), v'(x))$ restricted to  any of the edges lies on the positive heteroclinic orbit, the maximum of $v$ on $G$ has to be less than or equal to the maximum of $\bar v$, that is, $\bar v(0)$. The Kirkhoff conditions at the origin imply that either $v(0)= \bar v(0)= \(\frac{p}2\)^\frac1{p-2}$ or $N$ is necessarily even. Indeed, if $N$ is odd and $v(0)<\bar v(0)$, by Proposition  \ref{Prop:line},   there are two possible values for the outward derivatives at the origin following the edges of the graph, one positive, $d$, and one negative, $-d$. For any integer $q$ less than $N$, the sum of the outward derivatives at the origin is equal to $q d - (N-q) d = (2q-N) d \neq 0$, thus violating Kirkhoff conditions. So, the only possibility for $N>1$ odd is that $v(0)= \bar v(0)$ and the outward derivatives at the origin in the direction of all the edges are all equal to $0$. In this case, since $N\ge 3$,
\[
F_{G, p}(v)= \(\frac{N}2\)^{1-\frac2p} I_{p,0} > \frac{1}{C_p(\R)} \ge \frac{1}{C_p(G)} \,,
\] 
by Lemma \ref{constant-ineq}. Therefore, $v$ is the unique solution to the Kirkhoff-ODE system in $G$, but not a solution to \eqref{min}. 

If $N$ is even, $v(0)$ can take any value in the interval $(0, \bar v(0)]$ and with $x\in \R$ defined by $v(0)=\bar v(x)$,  the outward derivatives at the origin along the $N$  edges have to be equal to    $d=  \bar v'(x)$ or to $-d=-\bar v'(x)$. So, we can match the edges two by two, so that two by two they sustain a function  which is equivalent to $\bar v$. If $N>2$ then 
\[
F_{G, p}(v)= \(\frac{N}2\)^{1-\frac2p} I_{p,0} > \frac{1}{C_p(\R)} \ge \frac{1}{C_p(G)} \,,
\] 
by Lemma \ref{constant-ineq}.
In the above two cases, the non-existence of solutions to \eqref{min} shows that
$ C_p(\R) = C_p(G)$,
by Theorem \ref{Thm:conc-comp2}.   On the contrary, if $N=2$, for any value of $v(0)$ in the
interval $(0, \bar v(0)]$, $v$ will satisfy
\[
F_{G, p}(v)=  I_{p,0} = \frac{1}{C_p(\R)}  \,.
\] 
and therefore $v$ will be a minimizer for \eqref{min}.
In this case again by Lemma \ref{constant-ineq} \[
\frac{1}{C_p(\R)}= \frac{1}{C_p(G)}, \; \mbox{ that is, }\;\; C_p(\R) = C_p(G)
\]
and $v$ will be equivalent to $\bar v$ and a solution to \eqref{min}.

\end{proof}

\subsection{One infinity edge and two bounded ones}\label{Sec:15-infty}
Consider a graph consisting of three edges, one equivalent to a half-line, $e_0$ and two others, $e_1, e_2$, bounded, attached to $e_0$ on its finite endpoint.  In this case we can encounter very different situations. 

Without loss of generality, we can assume that $e_0= \R_+$ and denote by $\ell_1, \ell_2$  the lengths of the bounded edges. Let us now consider functions $v_i, i=0, 1,2$ satisfying the equation the Kirkhoff-ODE system in $G$ in their respective edges and such that $v_i(0)= v_0(0)$ for $i= 1, 2$. If $\ell_1$ is small enough, consider the case where $v=v_1$ is monotone along $e_1$. Then, the outward derivative at $0$ along $e_1$ has to be negative and  large enough in absolute value, so that  $v$ reaches $0$ at the outer end of the edge. Also, $|v_1'|$ will have to be large at the outer end of the edge $e_1$. If $\ell_2$ is large enough, we can find a function $v_2$ which is not monotone on $e_2$ and such that the outward derivative at the origin is positive and  not larger than $2 \max_s [\bar v'(s)|$, with $\bar v$ defined by \eqref{vbar}. Hence by taking $\ell_1$ small enough and $\ell_2$ large enough, we see that the sum of the outward derivatives at the origin of these three functions will be negative. 

Now, keeping $\ell_1$ and $\ell_2$ fixed, and varying continuously the values of $v_1'$ and $v_2'$ at the outer ends of their respective edges, or equivalently, the value of the functions $v_i$ at the origin, we can find new functions $v_0, v_1, v_2$ such that the sum of the outward derivatives at the origin will be equal to $0$. And this can be done in two different ways: either keeping $v_0$ monotone or not. Indeed, if we fix $v_i(0)=\epsilon$ small and we choose $v_0$ and $v_1$ to be monotone, we see that the sum of the outward derivatives at the origin will be positive. Indeed,  the outward derivatives of $v_0$ and $v_1$ will be negative and very small in absolute value. On the contrary $v_2$ will not be monotone, and  at the end of $e_2$ the outward derivative at the origin along $e_2$ will be close in absolute value to the value of $v_2'$ at the outer end of $e_2$. Hence, as $\epsilon$ goes to $0$, $v_2$ gets very close to the function ${\tilde v}$ solution to
$$
-{\tilde v}'' +{\tilde v} = {\tilde v}^{p-1} \;\mbox{ in } \;(0, \ell_2), \quad {\tilde v} >0 \;\mbox{ in } \;(0, \ell_2), \quad{\tilde v}(0)={\tilde v}(\ell_2)=0\,.
$$
As $\epsilon$ goes to $0$ the sum of the three outward derivatives at the origin will be equal to $ - {v_0}'(\ell_2)$, a positive number. Hence, by continuity, in the middle there must be functions $v_0, v_1, v_2$ such that at the origin they satisfy the Kirkhoff conditions and such that $v_1$ and $v_0$ are monotone in their respective edges, while $v_2$ is not.  

The same reasoning can be done when now we consider $v_1$ monotone, but $v_0$ and $v_2$ non monotone. Thus, we have found a different solution to the Kirkhoff-ODE system in $G$.

Therefore, by a careful consideration of the phase plane in Figure \ref{Fig:phase-portrait}, and by {\it playing} with the lengths of the edges $e_1$ and $e_2$, we have proved the following.
\begin{prop}
Let $p>2$. Assume that $G$ is a graph with three edges which meet at the origin, one of them being a half-line and the two others having lengths $\ell_1, \ell_2\in \R_+$. If $\ell_1$ is small enough and $\ell_2$ is large enough, there are at least two different non-negative solutions to the Kirkhoff-ODE system in $G$. 
\end{prop}

We see a numerical example of the above result in the following figure obtained for $\ell_1=1$ and $\ell_2=5$ (see Figure \ref{Fig:15-inf}).
\begin{figure}[ht]\label{Fig:15-inf}
\begin{center}
\includegraphics[width=5.5cm,height=5.5cm]{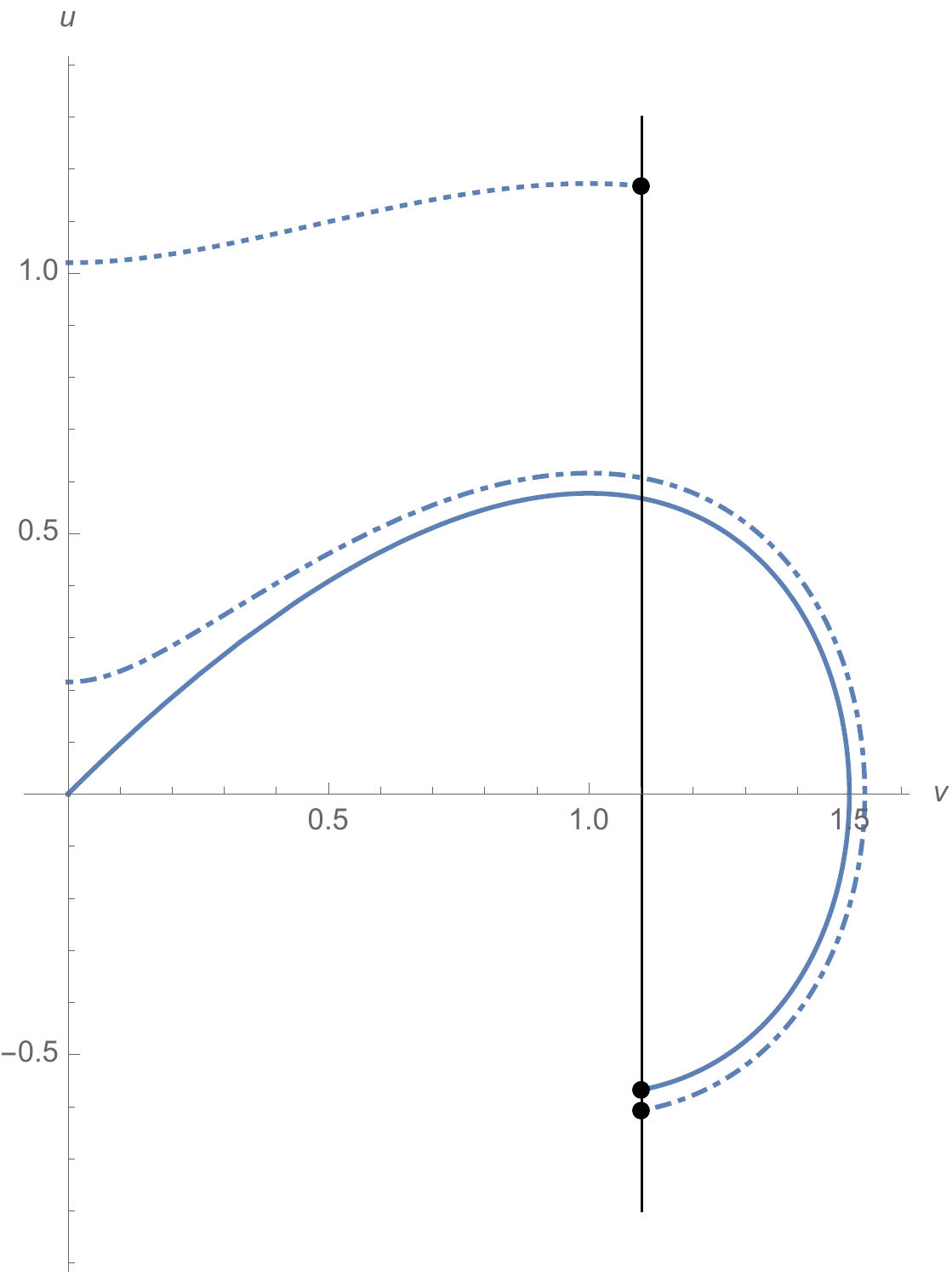} \hspace{14mm} \includegraphics[width=5cm,height=5cm]{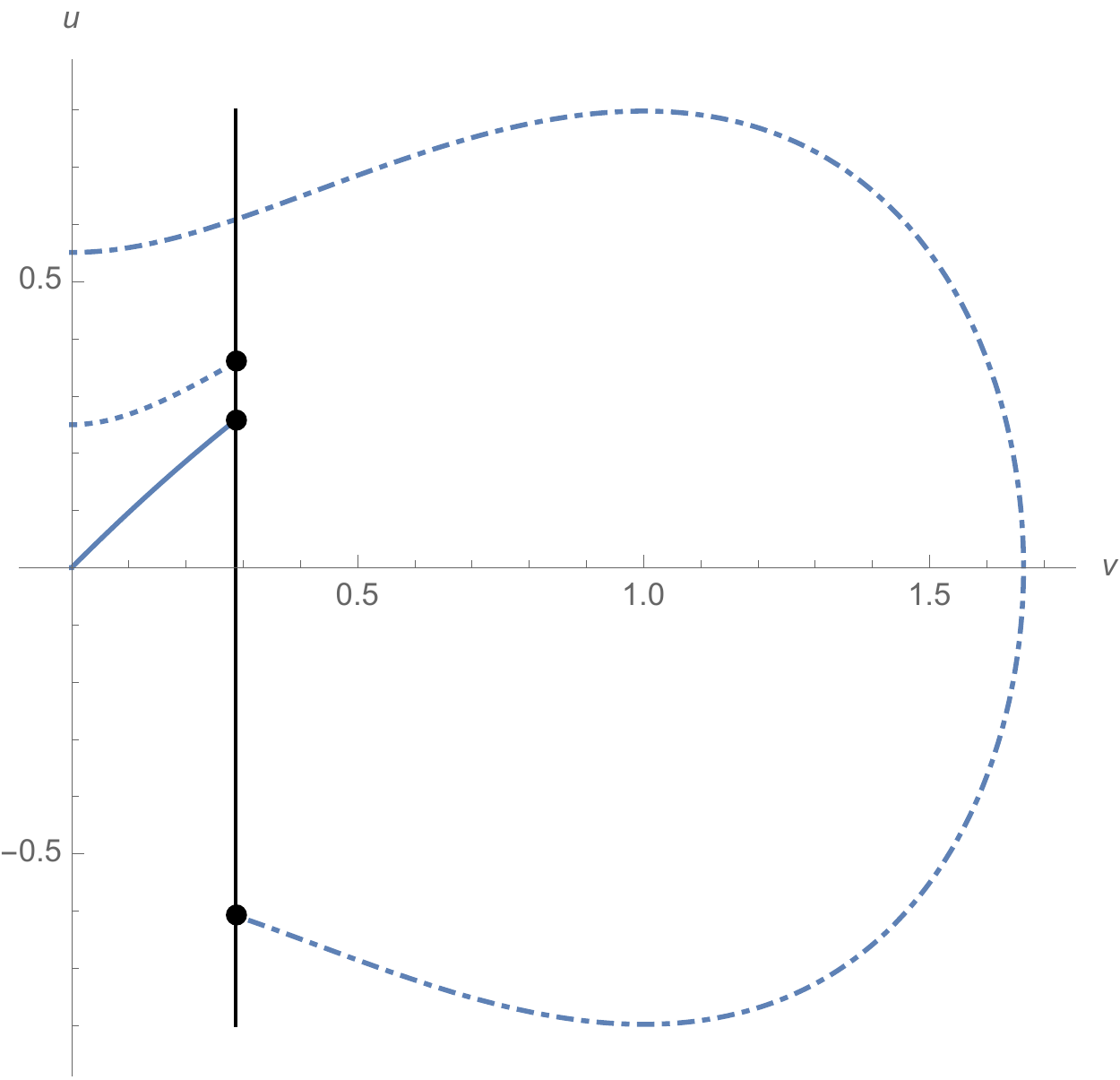} 
\caption{ Solution of the Kirkhoff-ODE system for $p=3$, and edges of lengths $1, 5$ and $+\infty$. The dotted curves correspond to the functions $v_1$, the dot-dashed to  $v_2$ and the full line to $v_0$. The dots in the graphics correspond in the phase plane to the values of $v_0, v_1$ and $v_2$ at the origin: the abscissas of the three dots coincide and the sum of the outward derivatives, corresponding to $-$ the sum of the ordinates of the dots, is equal to $0$.}
\end{center}
\end{figure}

\medskip
Let us now see a case of non-existence of solutions to the Kirkhoff-ODE system, and therefore neither to \eqref{min}.

Assume now that $\ell_1$ is small enough and fixed.  For simplicity consider $e_1$ as a copy of  the interval $(0, \ell_1)$. Let $v$ any monotone non-negative solution of 
\[
-v'' + v = |v|^{p-2}v  \;\mbox{ in }\; (0, \ell_1)\,, \quad v(0)=0\,,
\]
such that $ v(\ell_1)\le v_0(0)$. Call $M_1$ the maximum of $|v'(\ell_1)| $ for all those functions $v$. And call $M_0$ the maximum of $\bar v'(s)$ for all $ s\in \R$, with $\bar v$ defined by \eqref{vbar}.
Assume that there is a solution of the Kirkhoff-ODE system for this graph. Necessarily $v(0)\le v_0(0)$. For any value of $v(0)$, the sum of the outward derivatives along the edges $e_0$ and $e_1$ will be negative and less than $M_0+M_1$ in absolute value. Therefore the function $v_2$ must be non-monotone and its outward derivative must be positive in order to comply with the Kirkhoff conditions. Now, if $\ell_2$ is small enough, the outward derivative at the origin of any function $v_2$ such that at the origin $v_2$ is not larger than $ v_0(0)$ will be larger than $M_1+M_0$. A contradiction. 

Since for any given value of $v(0)$, $|v'_1(0)|$ and $|v'_1(0)|$ must be larger than $|v'_0(0)|$, if $v_1$ were not monotone, then we could inverse the roles of $v_1$ and $v_2$ to reach the same contradiction. We have thus proved the following.

\begin{prop}
Let $p>2$.  Assume that $G$ is a graph with three edges which meet at the origin, one of them being a half-line and the two others having lengths $\ell_1, \ell_2\in \R_+$. If both $\ell_1$ and $\ell_2$ are small enough, there is no solution of the Kirkhoff-ODE system in $G$. \end{prop}
\begin{remark} That under the  assumptions of the above proposition, there is no solution of \eqref{min} was already proved in Theorem \ref{Thm:one-unbdd}. But it is also a corollary of the above proposition.
\end{remark}

\subsection{Three bounded edges}\label{Sec:3-bounded}
This is a very easy case, because it is a direct consequence of Theorem \ref{Prop:compact-ms}. Indeed, all bounded locally finite graphs are compact sets. Therefore, we have the following.

\begin{cor}
For a three bounded edges' graph with a single inner vertex, there exists a non-negative solution of the minimization problem \eqref{min}, that is, a extremal function for the inequality \eqref{ineq}. That implies in particular, that there exists at least one solution of the equation the Kirkhoff-ODE system, but there could be more than one. See below examples of the two situations.
\end{cor}

In the first case for which we have computed explicit solutions, we take the lengths of the three edges to be equal to $1, 2$ and $5$. In this case we find numerically a solution to the Kirkhoff-ODE system. See Figure \ref{Fig:125} below.

\begin{figure}[ht]\label{Fig:125}
\begin{center}
\includegraphics[width=6cm,height=5.5cm]{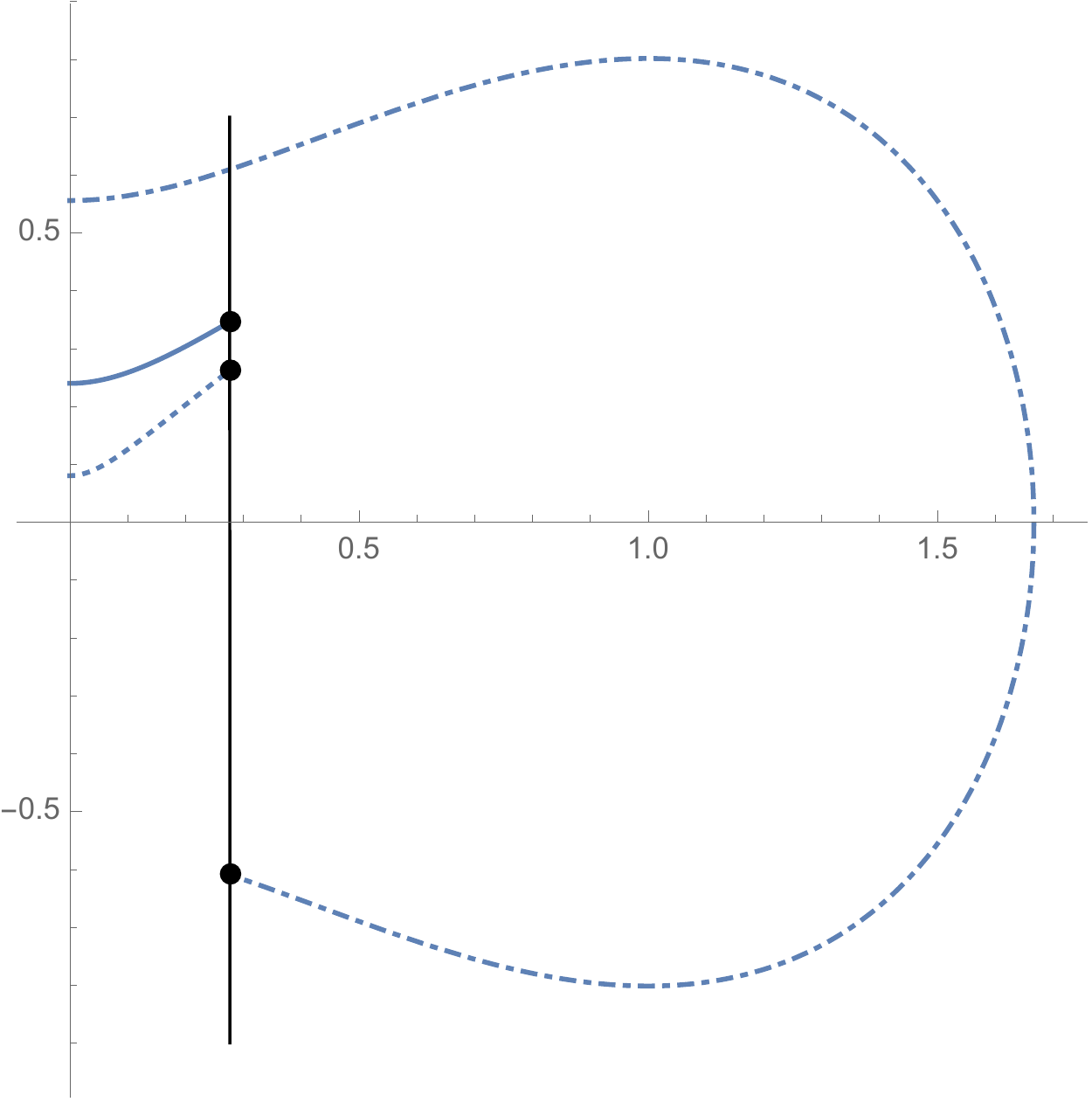}
\caption{\label{125} Solution of the Kirkhoff-ODE system for $p=3$, and edges of lengths $1, 2$ and $5$. The dotted curve corresponds to $v_1$ ($\ell_1=1$), the dot-dashed one to $v_3$ ($\ell_3=5$) and the full line to $v_2$ ($\ell_2=2$). We observe that this solution is near the solution to the $1, 5, +\infty$ problem on the right of Figure \ref{Fig:15-inf}. On the contrary, it is difficult to imagine a solution near the one on the left of Figure \ref{Fig:15-inf}  for $\ell_2=2$. Indeed, if $\ell_2$ were larger, we would be able to find one such solution, making the number of solutions larger than or equal to $2$. We see this in the next figure, where $\ell_2= 4$  }
\end{center}
\end{figure}

When we instead increase the length of the second edge, that is, we take the lengths of the three edges to be equal, for instance,  to $1, 4$ and $5$, we are able to find numerically two different solutions to the Kirkhoff-ODE system. See Figure \ref {Fig:145}.

\begin{figure}[ht]\label{Fig:145}
\begin{center}
\includegraphics[width=5.5cm,height=5.5cm]{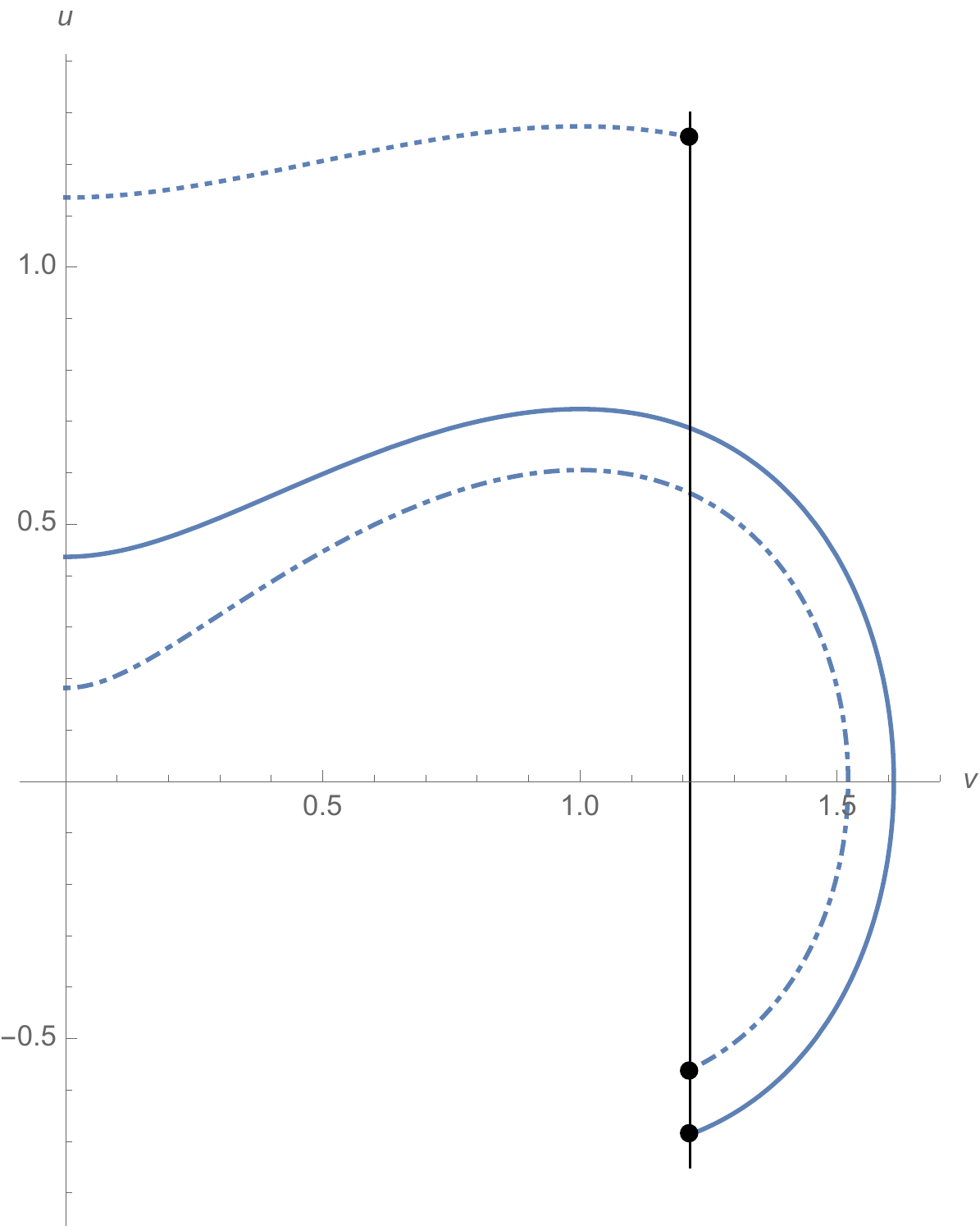} \hspace{14mm} \includegraphics[width=5cm,height=5cm]{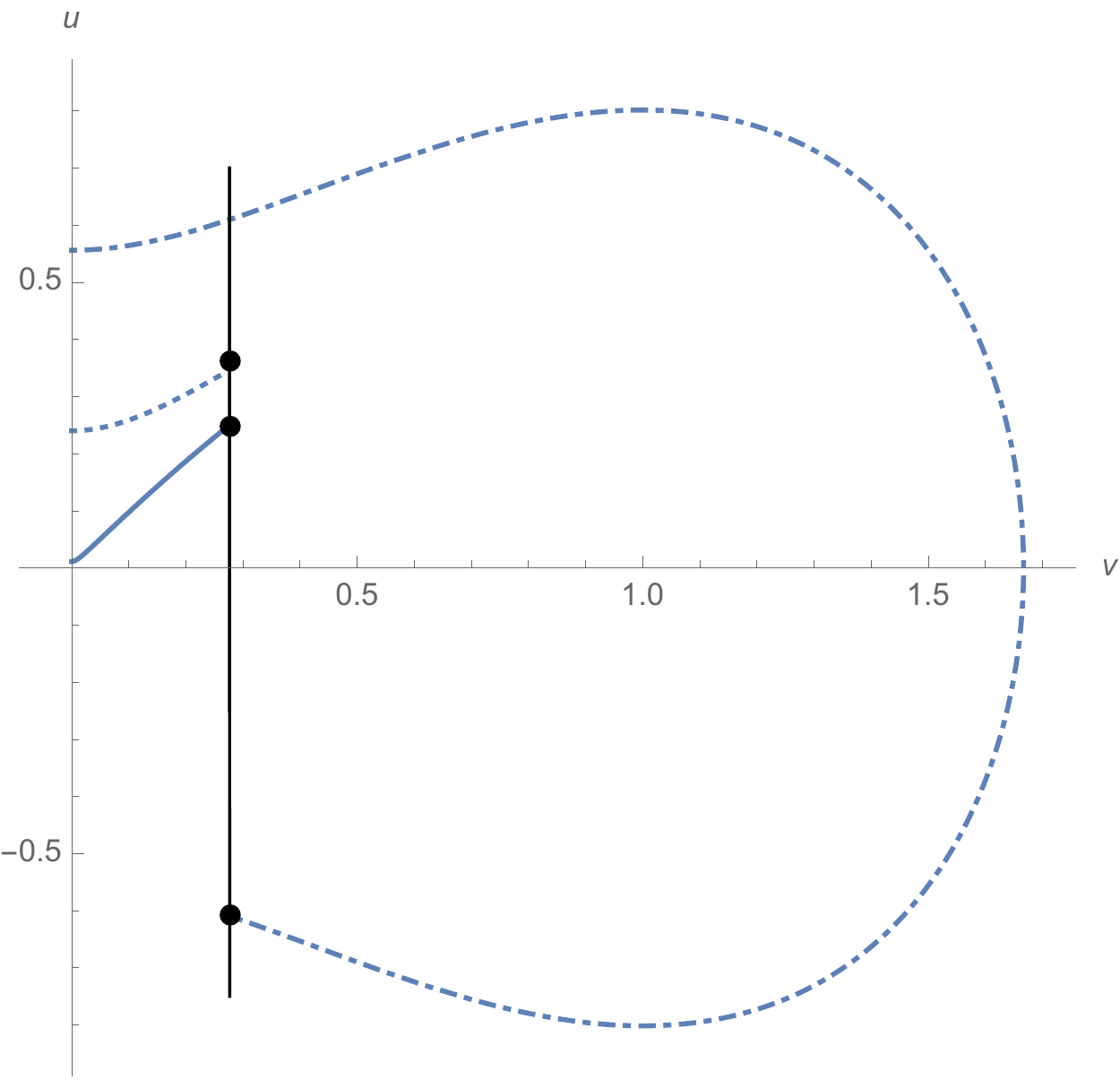} 
\caption{ Solution of the Kirkhoff-ODE system for $p=3$, and edges of lengths $1, 4$ and $5$. The dotted curves correspond to the functions $v_1$ ($\ell_1=1$), the dot-dashed to  $v_3$ ($\ell_3=5$) and the full line to $v_2$ ($\ell_2=4$). The dots in the graphics correspond to the origin : the values of $v_1(0), v_2(0)$ and $ v_3(0)$  (the abscissas of the three dots) coincide and the sum of the outward derivatives, corresponding to $-$ the sum of the ordinates of the dots, is equal to $0$.  Note that the figure on the left has the same structure as the left figure on Figure \ref{Fig:15-inf}, and the same analogy is found for the figures on the right.}
\end{center}
\end{figure}

\bibliographystyle{siam}
\bibliography{ineq-graphs}

\end{document}